\documentclass[12pt,a4paper]{article}
\usepackage{amssymb}
\usepackage{amsmath,amsfonts,amsthm}

\usepackage{graphicx}
\usepackage{amsmath}

\usepackage{amsfonts}
\usepackage{amsthm,amscd}
\usepackage{graphicx}
\usepackage{amsmath}
\usepackage{amssymb}
\usepackage{amstext}

\newtheorem{theorem}{Theorem}
\newtheorem{corollary}[theorem]{Corollary}
\newtheorem{definition}[theorem]{Definition}

\newtheorem{lemma}[theorem]{Lemma}

\newtheorem{proposition}[theorem]{Proposition}

\title{Selection of measure and a Large Deviation Principle for the general one-dimensional XY model}

\author{A. O. Lopes\,\thanks{arturoscar.lopes@gmail.com, Instituto de Matem\'atica - UFRGS - Partially supported by DynEurBraz, CNPq, PRONEX -- Sistemas
Dinamicos,  INCT, Convenio Brasil-Franca} \, and J. K. Mengue\thanks{jairokras@gmail.com,  Instituto de Matem\'atica - UFRGS} }

\begin{document}

\maketitle

\begin{abstract}

We consider $(M,d)$ a connected and compact manifold and we denote
by $X$ the Bernoulli space $M^{\mathbb{N}}$. The shift acting on $X$
is denoted by $\sigma$.

We analyze  the general XY model, as presented in a recent paper by
A. T.  Baraviera, L. M. Cioletti,  A. O. Lopes, J. Mohr  and R. R.
Souza. Denote the Gibbs measure by $\mu_{c}:=h_{c}\nu_{c}$, where
$h_{c}$ is the eigenfunction, and, $\nu_{c}$ is the eigenmeasure of
the Ruelle operator associated to $cf$. We will show that
any measure selected by $\mu_{c}$, as $c\to +\infty$, is a
maximizing measure for $f$. We also prove, when the maximizing probability
measure is unique, that a certain Large Deviation Principle holds with
the deviation function $R_{+}^{\infty}=\sum_{j=0}^\infty R_{+}
(\sigma^j )$, where $R_{+}:= \beta(f) + V\circ\sigma - V - f$ and
$V$ is any calibrated subaction.
\newline

\end{abstract}

\section{Introduction}

We consider $(M,d)$ a connected and compact finite dimensional manifold and we denote
by $X$ the Bernoulli space $M^{\mathbb{N}}$. The shift acting on $X$
is denoted by $\sigma$.

We point out that the number of preimages by $\sigma$ of each point is not countable.

Let $f: X \to \mathbb{R}$ be a fixed Holder potential defined in the
Bernoulli space $X$. We denote by $m$ the Lebesgue probability on
$M$. We suppose without loss of generality that the diameter of the manifold $M$ is smaller than one. This distance induces another one, in the usual fashion, on $M^{\mathbb{N}}$ \cite{BCLMS}.

We are interested in the Gibbs state (for finite and zero
temperature) associated to  the potential $f$. This model is called
the general $XY$ model in \cite{BCLMS}. We refer the reader to such
work for a detailed explanation about the motivation for considering
such kind of problems. We point out that in the literature in Physics what is called the XY model is the case when $M=S^1$, and, the potential $f$ depends just on a finite number of coordinates. In \cite{BCLMS} and here the hypothesis are more general.

Classical references in the $XY$ model are \cite{Isr}, \cite{Mayer} and \cite{VE} where the spin can be in a circle. A nice reference for general results in Statistical Mechanics is \cite{G}.

In order to define a transfer operator we need a probability a priori on $M$ which we will denote by $d\,m$. In the case $M=S^1$ is usual to consider the Lebesgue measure $dx$ (see \cite{VE}) as  the a priori probability.
In \cite{LMMS} it is consider the case of the one dimensional  spin lattice with  a general a priori probability and is  presented results which are generalizations of the ones in
\cite{BCLMS}.

First we will recall some definitions and results from \cite{BCLMS}.

\begin{definition}

Let  $\mathcal{C}$ be the space of continuous functions from
$X=M^\mathbb{N}$ to $\mathbb{R}$. We define  the Ruelle operator $L_f$ acting on
$\mathcal{C}$,  associated to the Holder potential $f:M^\mathbb{N}
\to \mathbb{R}$, as the linear operator that  takes $w \in
\mathcal{C}$, and send it  to $L_f(w) \in \mathcal{C}$, which is defined for any
$x=(x_0,x_1,x_2,....)\in X$, in the following way
$$L_f(w)(x)=\int e^{f(ax)}\,w(ax)\, dm\,(a)\,,
$$
where $ax$ represents the sequence $(a,x_0,x_1,x_2,....)\in X$, and
$dm\,(a)$ is the Lebesgue probability on $M$.
\end{definition}

Following \cite{BCLMS}, for a real value $c$ we consider $\beta_{c}$
the main eigenvalue, $h_{c}$ the associated  eigenfunction, and $g_{c} = cf +
\log(h_{c}) - \log(h_{c}\circ\sigma) -\log(\beta_{c})$ the
normalized function associated to the Ruelle operator $L_{c\,f}$
obtained from $cf$. We also denote by $\nu_{c}$ the eigenmeasure of
$L_{c\,f}^*$, and, by $\mu_{c}:=h_{c}\nu_{c}$, the Gibbs probability of
the potential $c f$.

Note that if for each point there exists an uncountable number of preimages, then it is necessary  to use an a priori  probability $m$ in Definition 1.

As usual, by notation $f^n(x)=\sum_{j=0}^{n-1} f(\sigma^j(x)),$  for
any $n\in \mathbb{N}$, $x \in X$.

\medskip

\textbf{Remark on notation:} the iterated Ruelle Operator $L_{f}^{n} w(x), \,
n=1,2,3...$, can be written as
\[\int_{a_{n}...a_{1}}e^{f^{n}(a_{n}...a_{1}x)}w(a_{n}...a_{1}x) \, da_{1}...da_{n}, \ \ or \ \  \int_{\sigma^{n}z=x}e^{f^{n}(z)}w(z)\, dm.\]

\medskip

The Bernoulli space $M^{\mathbb{N}}$ has the structure of a differentiable infinite dimensional manifold.
In this new setting of Thermodynamic Formalism one can ask different kinds of questions: if the potential is differentiable, is it true that there exists
a positive differentiable main eigenfunction. This question was addressed  in \cite{LMMS}. Our focus here is in a different class of problems.

 We denote by
$\mathcal M_\sigma$ the compact set of invariant probability measures for
$\sigma$.

We  consider the following problem: for a given $f:X \to
\mathbb{R}$, we want to find probability measures which  maximize, over $
\mathcal M_\sigma$, the value $ \int f(x) \,d \mu(\mathbf {x}). $

\begin{definition}\label{max1}
We define  
$$\beta(f)=\max_{\mu\in\mathcal M_\sigma} \left\{ \int f d\mu \right\}\,.$$

Any of the measures which attains the maximal value will be called a
maximizing probability measure for $f$, which is sometimes denoted by
$\mu_\infty$.

\end{definition}

There exist Holder potentials $f$ such that the maximizing probability is not unique. In case of the shift acting on the Bernoulli space
$\{1,2,.,d\}^\mathbb{N}$, it is known that for a generic potential $f$  in the Holder norm the maximizing probability is unique (see \cite{CLT}). It can be for instance a probability with support in a periodic orbit. A basic reference for the classical study of maximizing probabilities is \cite{BLL1}.

It is known that for any fixed $x\in M^{\mathbb{N}}$ the probabilities $\mu_{n,c}$ defined by
\[\int\,w\,d\mu_{n,c} \,:= L_{g_c}^{n}(w)(x)\]
converges to $\mu_c$ in the weak* topology, when, $n\to\infty$ (see for instance  \cite{BCLMS}).

The Classical Thermodynamic Formalism setting
considers the shift acting on the Bernoulli space
$\{1,2,.,d\}^\mathbb{N}$ (see \cite{PP}).
In this case one can consider the classical Kolmogorov
entropy and pressure.  In this way, the  variational principle can be used to prove that any limit of the the Gibbs measure $\mu_c$ (when $c\to\infty$) is a maximizing measure to $f$ (see Prop. 29 in  \cite{CLT}
or \cite{CoG}).
In order to prove this last result in the present setting we use a different approach: we consider the double limit\footnote{Given a double indexed sequence $z_{c,n}$, $c\in \mathbb{R}$,
$n \in \mathbb{N}$, we say that $\lim_{c,n\to\infty} z_{c,n}=w$, if for any
given $\epsilon>0$, there exists an $M>0$, such that, if $c,n>M$, then
$|z_{c,n}- w|<\epsilon$.
} for $\mu_{n,c}$.

In Section 2 we will show the following:
\begin{theorem}\label{theo1}
Any weak$^*$-limit of subsequence of $\mu_{c}$, $(c\to +\infty)$, is
a maximizing probability measure for $f$. If the maximizing probability $\mu_\infty$ for $f$ is unique, then $\mu_c\to \mu_\infty$, when $c \to \infty$.
\end{theorem}

It is known that the parameter $c$ corresponds to the inverse of temperature in Statistical Mechanics.
In this way one can say that any convergent subsequence of Gibbs
states at positive temperature, when $c\to \infty$, selects at zero temperature maximizing probabilities. For this result we do not assume uniqueness of the maximizing probability.

  A ground state is by definition a maximizing probability that can be ``selected" as the limit of $\mu_{c_n}$ for some subsequence $c_n \to \infty$. There are examples of Holder potentials $f$ where certain maximizing probabilities are not ground states (see \cite{Bremont}, \cite{leplaideur-max}, \cite{BLL}, \cite{BLM}).
From the above we realize that maximizing probabilities are related to the concept of ``Gibbs state at temperature zero".

In the present case a definition of entropy is possible and this is described in detail in \cite{LMMS}. 
In this case the entropy is a non positive number (see also \cite{LNT}).
The main conclusions one can get from \cite{LMMS} is that: I) In the concept of Kolmogorov entropy there exist an a priori measure which is hidden; II)
Entropy and Ruelle operator are concepts which are linked.

Here we will not take advantage of the approach described in \cite{LMMS} and our proof
will use other more simple methods, where we consider the double limit for $L_{gc}^{n}(\cdot)(x)$.

\begin{definition}\label{sub}
A continuous function $V: X \to \mathbb{R} $ is called a {\em
calibrated subaction} for $f:X\to \mathbb{R}$, if for any $y\in X$,
we have that

\begin{equation}\label{c} V(y)=\max_{\sigma(x)=y} [f(x)+ V(x)-\beta(f)].\end{equation}

\end{definition}

This can be also be expressed as
$$\beta(f)= \max_{a \in M} \{f(ay)+  V (ay) - V(y) \} .$$

One can show that for any $x$ in the support of a maximizing probability measure for $f$ we have that
$$ V(\sigma(x)) -  V(x) - f(x) + \beta(f)=0.$$ In this way if we know the value $\beta(A)$, then a calibrated subaction $V$ for $f$ helps to identify the union of the supports of maximizing probabilities $\mu_\infty$ for $f$. The above equation can be eventually true outside the union of the supports of the  maximizing probabilities  $\mu$.

It is known that $\frac{1}{c} \log (h_{c})$, $c \in
\mathbb{R}$, is a equicontinuous family. Any limit of subsequence
$V= \lim_{n \to \infty} \frac{1}{c_n} \log (h_{c_n})$, $c_n \to
\infty$, is a calibrated subaction (see \cite{BCLMS}).
If the maximizing probability is unique, then the calibrated
subaction is unique up to an additive constant (see \cite{BLT}
 \cite{GL1}). In this way we write $R_{+}^{\infty}=\sum_{j=0}^\infty R_{+}
(\sigma^j )$, with  $R_{+}:= \beta(f) + V\circ\sigma - V - f$, where
$V$ is any chosen  calibrated subaction.

In section 3 we consider the case where the maximizing measure for
$f$ is unique, and prove that the family $L_{g_c}^{n}(\cdot)(x)$ satisfies the
following kind of Large Deviation Principle:

\begin{theorem}\label{ldpruelle1}
Suppose that the maximizing probability for $f$ is unique. Consider any
point $x\in X$, then, for any closed set $F$, and any open set $A$:
\[ \limsup_{c,n\to\infty} \frac{1}{c}\log((L_{g_{c}}^{n}\chi_F)(x))
\leq - \inf_{z\in F} (R_{+}^{\infty}(z)),\]
\[ \liminf_{c,n\to\infty} \frac{1}{c}\log((L_{g_{c}}^{n}\chi_A)(x))
\geq - \inf_{z\in A} (R_{+}^{\infty}(z)).\]
The function $I=R_{+}^{\infty}$ is lower semicontinuous, non negative and zero at least in the support of the maximizing probability.
\end{theorem}

The estimation of the asymptotic above via the Ruelle operator is useful because the computations of most of the
important objets in Thermodynamic Formalism are obtained via this operator.

From the last theorem, taking $x$ fixed, $n\to\infty$, and using the fact that $L_{g_c}^{n}(\cdot)(x)$ converges to $\mu_c$ weakly* (see \cite{BCLMS}) we get the following:

\begin{corollary}
Suppose $f: (\mathbb{S}^{1})^\mathbb{N}\to \mathbb{R}$ is  a fixed Holder potential defined in the
Bernoulli space. For any cylinder $F=I_1 \times...\times I_n \times \mathbb{S}^{1}\times  \mathbb{S}^{1}\times \mathbb{S}^{1}\times....,$ where each $I_j$ is a closed interval of $\mathbb{S}^{1}$ we have
\[ \limsup_{c\to\infty} \frac{1}{c}\log(\mu_c(F))
\leq - \inf_{z\in F} (R_{+}^{\infty}(z)).\]
For any cylinder $A=I_1 \times...\times I_n \times \mathbb{S}^{1}\times  \mathbb{S}^{1}\times \mathbb{S}^{1}\times....,$ where each $I_j$ is a open interval of $\mathbb{S}^{1}$ we have
\[ \liminf_{c\to\infty} \frac{1}{c}\log(\mu_c(A))
\geq - \inf_{z\in A} (R_{+}^{\infty}(z)).\]

\end{corollary}

We point out that the reasonings which we use in the proofs of the above results can
also be applied to the  Classical Thermodynamic Formalism setting
\cite{PP}, where the Bernoulli space is $\{1,2,...,d\}^\mathbb{N},$
to get the analogous result presented in \cite{BLT}. The present proof of the L. D. P. does not use
the involution kernel as in \cite{BLT}.

We say that the potential $f$ depends on two variables if for any $x=(x_0,x_1,x_2,x_3....)\in X$ we have that the value  $f(x_0,x_1,x_2,x_3....)$ is independent of $(x_2,x_3,x_4,....)$. This case is also known as the ``nearest-neighbour interaction". The last corollary was shown to be true in the case the potential $f$ depends on two coordinates in \cite{LMST}. We point out that our Corollary 6 is for a general Holder function $f$.

\medskip

There are two kinds of Large Deviations which are more commonly study:

I) Large deviations in time - In this case one is interested in results like: given an ergodic invariant measure $\mu$ for the shift
(or another  dynamical system) acting on $\{1,2,.,d\}^\mathbb{N}$, then,   what is the probability on $x$ that a given Birkhoff sum $ \frac{1}{n} \sum_{j=0}^{n-1} \varphi(\sigma^j (x)) $ deviates of the mean value $\int \varphi\, d \mu$ by, let's say, an error of $\epsilon$? This is what \cite{Ellis} call Level 1 L. D. P..

General references for this kind of dynamical results are for instance  \cite{LSY}, \cite{Ki1}, \cite{L2}, \cite{L7}  and \cite{L3}. These kind of problems can be studied in a more broad sense for general stochastic processes (see \cite{Ellis}, \cite{DZ}). The main point in the ``dynamical setting of expanding maps and Gibbs probabilities for Holder potentials" is that one can exhibit in an explicit form the deviation function (via Legendre transform of the pressure) which is analytic.  Differentiability of pressure plays an important role in all this.

In \cite{SSS} the authors show that the Ruelle operator is an analytic function of the Holder potential and this result can be applied to the present situation. Therefore, we believe it is just a question of rewriting the classical arguments of Thermodynamic Formalism on our setting  in order to show that
indeed is true a Large Deviation principle in time for the XY model.

\medskip

II) Large deviations for measures indexed by a parameter $c\in \mathbb{R}$ - General references for this kind of problem are \cite{St}(see section 0. Introduction) and \cite{DZ}. In this last  book in section 4.3 the parameter $\epsilon\to 0$ corresponds here to $c=1/\epsilon \to \infty$.

Consider a family of probabilities $\mu_\epsilon$ such that, $\lim_{\epsilon \to 0}\mu_\epsilon=\mu$. Given a set $J$ such that $\mu(J)=0$ one can
ask about the exponential velocity such that $\mu_\epsilon(J)$ goes to zero, as $\epsilon \to 0.$ This is the other kind of Large Deviation which is also very much analyzed.

The famous theorem of Schilder concerns a renormalization of time of the Probability of the Brownian motion by a small parameter $\epsilon$ (see Theorem 5.2.3 in \cite{DZ}). In Aubry-Mather theory this result is applied to the study of viscosity solutions, subactions   and Large Deviations  associated to diffusions (see Appendix in \cite{A2} and \cite{A1}).

We point out that here (in a similar way as in \cite{BLT}) we are able to express the deviation function $I$ in an explicit form from the information given by the calibrated subaction.
This deviation function $I$ is quite irregular: it is just lower semicontinuous and equal to $-\infty$ in most of the points. The bottom line is: L. D. P in   II) is quite different from I).

\medskip

Most  of the  ideas of the present paper are generalizations of results which are contained in \cite{M} where was considered just  the classical case of the shift acting on  $\{1,2,..,d\}^\mathbb{N}.$

In \cite{LM} is presented another kind of Large Deviation Principle:
the setting of zeta measures. In this case the proof does not require
that the maximizing probability is unique.

\section{The selection of measure}

\begin{lemma}
Let $V$ be a calibrated subaction, such that,  $V= \lim_{c \to
\infty} \frac{1}{c} \log (h_{c})$, and, $R_{-} = f + V -
V\circ\sigma - \beta(f)$, which is the limit function of the
$g_{c}/c$ associated.
For each $\epsilon>0$, there exists a constant $\psi_\epsilon$, such that,
for any $x \in X$
$$m(\{a\in M: R_{-}(ax) >-\epsilon\})>\psi_\epsilon>0.$$
\end{lemma}

\begin{proof}
Suppose $g_{c}$ converges to $R_{-}$, and then  write $g_{c}=cR_{-}
+\delta_{c}$, where $|\delta_{c}|_{\infty}/c \to 0$. Using the fact that $V$
is a calibrated subaction, we have that $R_{-}\leq 0$.

We fix $\epsilon >0$, and we define
\[A_{\epsilon}:=\{a: R_{-}(ax)\leq -\epsilon\}\]
\[B_{\epsilon}:=\{a: R_{-}(ax) > -\epsilon\}.\]
$V$ is Holder, so $R_{-}$ is Holder, then, it is a continuous
function on the first symbol. In this way, $A_{\epsilon}$ and
$B_{\epsilon}$ are mensurable sets. Then, we get
\[1 = L_{g_{c}}1(x) = \int e^{g_{c}(ax)} da = \int e^{cR_{-}(ax) + \delta_{c}(ax)} da.\]
Therefore,
\[1 = \int_{A_{\epsilon}} e^{cR_{-}(ax) + \delta_{c}(ax)} da + \int_{B_{\epsilon}} e^{cR_{-}(ax) + \delta_{c}(ax)} da\]
\[\leq \int_{A_{\epsilon}} e^{-c\epsilon + \delta_{c}(ax)} da + \int_{B_{\epsilon}} e^{0 + \delta_{c}(ax)} da\]
\[\leq \int_{A_{\epsilon}} e^{-c\epsilon + |\delta_{c}|_{\infty}} da + \int_{B_{\epsilon}} e^{0 + |\delta_{c}|_{\infty}} da\]
\[=e^{-c\epsilon + |\delta_{c}|_{\infty}} m(A_{\epsilon}) + e^{|\delta_{c}|_{\infty}}m(B_{\epsilon})\]
\[\leq e^{-c\epsilon + |\delta_{c}|_{\infty}}  + e^{|\delta_{c}|_{\infty}}m(B_{\epsilon}).\]
Let $c_{0}>0$ be such that $e^{-c_{0}\epsilon +
|\delta_{c_{0}}|_{\infty}}\leq 1/2$. Then, it follows that
\[1/2 \leq  e^{|\delta_{c_{0}}|_{\infty}}m(B_{\epsilon}),\]
so,
\[m(B_{\epsilon}) \geq \frac{1}{2e^{|\delta_{c_{0}}|_{\infty}}}.\]

Then, we just take $\psi_\epsilon =
\frac{1}{3e^{|\delta_{c_{0}}|_{\infty}}}$, and the result follows.
\end{proof}

\textbf{Proof of  Theorem $\ref{theo1}$}

\begin{proof}
Let $\nu$ be an accumulation point of $\mu_{c}$ given by the limit of a certain
subsequence $c_{j}\to \infty$. Let $c_{i}$ be a subsequence obtained from  this first
one such that there exists the limit $V$ of the sequence
$\frac{1}{c_i}\log h_{c_i}$. Let $R_{-}:= f+V-V\circ\sigma -
\beta(f)$ be the function associated to such limit. Then, we get
$g_{c_{i}}/c_{i} \to R_{-}$. Define $a:=\lim_{c_{i}\to\infty}
\mu_{c_{i}}(R_{-}) = \lim_{c_{j}\to\infty} \mu_{c_{j}}(R_{-})$.
Then, it follows that $a\leq 0$. We will  show that $a\geq
0$. More precisely we are going to show that for any fixed $x\in X$:
\[\liminf_{i,n \to\infty} L_{g_{c_{i}}}^{n}(R_{-})(x) \geq 0.\]

\bigskip

We write $\log h_{c_{i}} = c_{i}V + \delta_{i}$, where
\[\frac{|\delta_{i}|_{\infty}}{c_{i}} \to 0 .\]

We also write
\[L_{g_{c_{i}}}^{n}(R_{-})(x) = \int_{a_{n}...a_{1}}e^{g_{c_{i}}^{n}(a_{n}...a_{1}x)}R_{-}(a_{n}....a_{1}x) da_{1}...d_{a_{n}}\]
\[ = \frac{\int_{a_{n}...a_{1}} e^{c_{i}f^{n}(a_{n}...a_{1}x) +\log (h_{c_{i}}(a_{n}...a_{1}x)) - \log (h_{c_{i}}(x))}R_{-}(a_{n}...a_{1}x)da_{1}...da_{n}}{\int_{a_{n}...a_{1}}e^{c_{i}f^{n}(a_{n}...a_{1}x) +\log (h_{c_{i}}(a_{n}...a_{1}x)) - \log (h_{c_{i}}(x))}da_{1}...da_{n}}\]
\[= \frac{\int_{a_{n}...a_{1}}e^{c_{i}f^{n}(a_{n}...a_{1}x) +\log (h_{c_{i}}(a_{n}...a_{1}x)) }R_{-}(a_{n}...a_{1}x) da_{1}...da_{n}}{\int_{a_{n}...a_{1}}e^{c_{i}f^{n}(a_{n}...a_{1}x) +\log (h_{c_{i}}(a_{n}...a_{1}x))} da_{1}...da_{n} }\]
\[= \frac{\int_{a_{n}...a_{1}}e^{c_{i}f^{n}(a_{n}...a_{1}x) +c_{i}V(a_{n}...a_{1}x) + \delta_{i}(a_{n}...a_{1}x) -c_{i}V(x) -c_{i}\beta(f)}R_{-}(a_{n}...a_{1}x) da_{1}...da_{n}}{\int_{a_{n}...a_{1}}e^{c_{i}f^{n}(a_{n}...a_{1}x) + c_{i}V(a_{n}...a_{1}x) + \delta_{i}(a_{n}...a_{1}x) - c_{i}V(x) - c_{i}\beta(f)} da_{1}...da_{n}}\]
\[= \frac{\int_{a_{n}...a_{1}}e^{c_{i}R_{-}^{n}(a_{n}...a_{1}x) + \delta_{i}(a_{n}...a_{1}x)  }R_{-}(a_{n}...a_{1}x) da_{1}...da_{n}}{\int_{a_{n}...a_{1}x}e^{c_{i}R_{-}^{n}(a_{n}...a_{1}x) + \delta_{i}(a_{n}...a_{1}x) } da_{1}...da_{n}}.\]

For a fixed $\varepsilon>0$, we define the sets:
\begin{align*}
A_{n} &:= \{a_{n}...a_{1} : R_{-}(a_{n}...a_{1}x)<-\varepsilon\},\\
B_{n} &:= \{a_{n}...a_{1}: R_{-}(a_{n}...a_{1}x)\geq -\varepsilon\}, \\
C_{n} &:= \{a_{n}...a_{1}: R_{-}(a_{n}...a_{1}x)>-\frac{\epsilon}{2}\}.
\end{align*}

Clearly, we have that $C_{n}\subseteq B_{n}$. \newline

As $V$ is a calibrated subaction, then $C_{n}$ is not empty. We
remark that $A_{n}\cup B_{n} = M^{n}$, and, by the Lemma above,
for each $a_{n-1}...a_{1}$, we have that $m\{a_{n}: a_{n}...a_{1}
\in C_{n}\}>\psi_{\epsilon/2}>0$.

Then, we get
\[\int_{A_{n}}e^{c_{i}R_{-}^{n}+\delta_{i}} da_{1}...da_{n}= \int_{A_{n}}e^{c_{i}R_{-} +\delta_{i}}e^{c_{i}R_{-}^{n-1}\circ\sigma  }da_{1}...da_{n}\]
\[\leq \int_{A_{n}}e^{-c_{i}\varepsilon + |\delta_{i}|_{\infty}}e^{c_{i}R_{-}^{n-1}\circ\sigma }da_{1}...da_{n}\]
\[= e^{-c_{i}\varepsilon + |\delta_{i}|_{\infty}} \int_{A_{n}}e^{c_{i}R_{-}^{n-1}\circ\sigma}da_{1}...da_{n}\]
\[\leq e^{-c_{i}\varepsilon + |\delta_{i}|_{\infty}}\int_{M^{n}}e^{c_{i}R_{-}^{n-1}\circ\sigma}da_{1}...da_{n}\]
\[= e^{-c_{i}\varepsilon + |\delta_{i}|_{\infty}}\int_{M^{n-1}}e^{c_{i}R_{-}^{n-1}}da_{1}...da_{n-1},\]

and,

\[\int_{B_{n}}e^{c_{i}R_{-}^{n}+\delta_{i}}da_{1}...da_{n} = \int_{B_{n}}e^{c_{i}R_{-}+\delta_{i}}e^{c_{i}R_{-}^{n-1}\circ\,\sigma}da_{1}...da_{n}\]
\[\geq \int_{C_{n}}e^{c_{i}R_{-}+\delta_{i}}e^{c_{i}R_{-}^{n-1}\circ\,\sigma}da_{1}...da_{n}\]
\[\geq \int_{C_{n}}  e^{-c_i \frac{\epsilon}{2} \,-|\delta_{i}|_{\infty}}e^{c_{i}R_{-}^{n-1}\circ\,\sigma}da_{1}...da_{n}\]
\[\geq e^{-c_i \frac{\epsilon}{2} -|\delta_{i}|_{\infty}}\int_{C_{n}}e^{c_{i}R_{-}^{n-1}\circ\,\sigma}da_{1}...da_{n}\]
\[\geq e^{-c_i \frac{\epsilon}{2} -|\delta_{i}|_{\infty}} \int_{M^{n-1}}\int_{\{a_{n}\,:a\,_{n}...a_{1} \in \,\, C_{n}\}}e^{c_{i}R_{-}^{n-1}\circ\,\sigma}da_{1}...da_{n}\]
\[= e^{-c_i \frac{\epsilon}{2} -|\delta_{i}|_{\infty}} \int_{M^{n-1}}e^{c_{i}R_{-}^{n-1}}\int_{\{a_{n}\,:\,a_{n}...a_{1} \in \,\, C_{n}\}}da_{1}...da_{n}\]
\[= e^{-c_i \frac{\epsilon}{2} -|\delta_{i}|_{\infty}} \psi_{\epsilon/2} \int_{M^{n-1}}e^{c_{i}R_{-}^{n-1}}da_{1}...da_{n-1}.\]

It follows that

\[0 \leq \liminf_{i,n\to\infty} \frac{\int_{A_{n}}e^{c_{i}R_{-}^{n}+\delta_{i}}da_{1}...da_{n}}{\int_{B_{n}}e^{c_{i}R_{-}^{n}+\delta_{i}}da_{1}...da_{n}} \]
\[\leq \limsup_{i,n\to\infty} \frac{\int_{A_{n}}e^{c_{i}R_{-}^{n}+\delta_{i}}da_{1}...da_{n}}{\int_{B_{n}}e^{c_{i}R_{-}^{n}+\delta_{i}}da_{1}...da_{n}} \]
\[\leq \limsup_{i,n\to\infty} \frac{e^{-c_{i}\varepsilon + |\delta_{i}|_{\infty}}}{\psi_{\epsilon/2} e^{-c_i \frac{\epsilon}{2}-|\delta_{i}|_{\infty}}}\]
\[\leq \limsup_{i,n \to\infty} e^{ -c_{i}\,\varepsilon/2\,+ 2|\delta_i|_{\infty}}\psi^{-1}_{\epsilon/2}\]
\[=\limsup_{i,n\to\infty} e^{c_{i}\, ( -\epsilon/2\,+  2 \frac{|\delta_i|_{\infty}}{c_{i}})}\psi^{-1}_{\epsilon/2} =0.\]

In the same way

\[0 \leq \liminf_{i,n\to\infty} \frac{\int_{A_{n}}e^{c_{i}R_{-}^{n}+\delta_{i}}R_{-}\,\,da_{1}...da_{n}}{\int_{B_{n}}e^{c_{i}R_{-}^{n}+
\delta_{i}}da_{1}...da_{n}(-\varepsilon)} \]
\[\leq \limsup_{i,n\to\infty} \frac{\int_{A_{n}}e^{c_{i}R_{-}^{n}+\delta_{i}}R_{-}\,\,da_{1}...da_{n}}{\int_{B_{n}}e^{c_{i}R_{-}^{n}+\delta_{i}}da_{1}...da_{n}(-\varepsilon)} \]
\[\leq \limsup_{i,n\to\infty} \frac{\int_{A_{n}}e^{c_{i}R_{-}^{n}+\delta_{i}}da_{1}...da_{n}\,(-|R_{-}|_{\infty})}{\int_{B_{n}}e^{c_{i}R_{-}^{n}+\delta_{i}}da_{1}...da_{n}(-\varepsilon)} \]
\[\leq \limsup_{i,n\to\infty} \frac{e^{-c_{i}\varepsilon + |\delta_{i}|_{\infty}}(-|R_{-}|_{\infty})}{\psi_{\epsilon/2} (-\varepsilon)e^{-c_i \, \epsilon/2\,-\,|\delta_{i}|_{\infty}}} = 0.\]

From the above,  and writing $\int_{M^{n}}da_{1}...da_{n} =
\int_{A_{n}}da_{1}...da_{n}+\int_{B_{n}}da_{1}...da_{n}$, we have:

\[\liminf_{c_{i},n\to\infty} L_{g_{c_{i}}}^{n}(R_{-})(x) = \liminf_{c_{i},n\to\infty}\frac{\int_{a_{n}...a_{1}}e^{c_{i}R_{-}^{n}(a_{n}...a_{1}x) + \delta_{i}(a_{n}...a_{1}x)  }R_{-}(a_{n}...a_{1}x) da_{1}...da_{n}}{\int_{a_{n}...a_{1}}e^{c_{i}R_{-}^{n}(a_{n}...a_{1}x) + \delta_{i}(a_{n}...a_{1}x) } da_{1}...da_{n}}\]
\[\geq \liminf_{c_{i},n\to\infty} \frac{\int_{A_{n}}e^{c_{i}R_{-}^{n} + \delta_{i}  }R_{-} da_{1}...da_{n}+ \int_{B_{n}}e^{c_{i}R_{-}^{n} + \delta_{i}  }da_{1}...da_{n}(-\varepsilon)}{\int_{A_{n}}e^{c_{i}R_{-}^{n} + \delta_{i}  }da_{1}...da_{n} + \int_{B_{n}}e^{c_{i}R_{-}^{n} + \delta_{i}  }da_{1}...da_{n}}\]
\[= \liminf_{c_{i},n\to\infty}\frac{\int_{B_{n}}e^{c_{i}R_{-}^{n} + \delta_{i}  }da_{1}...da_{n}(-\varepsilon)}{\int_{B_{n}}e^{c_{i}R_{-}^{n} + \delta_{i}  }da_{1}...da_{n}}\]
\[\geq -\varepsilon.\]
Taking $\varepsilon \to 0$, we get our claim.
\end{proof}

This ends the proof of our first main result. The bottom line is: a convergent subsequence of Gibbs states at
positive temperature selects maximizing probabilities (eventually, different  subsequences can localize different probabilities).

\section{On the Large Deviation Principle}

On this section we are going to prove Theorem $\ref{ldpruelle1}.$

We suppose that the maximizing measure for $f$ is unique, and we denote by $\mu_{\infty}$ the maximal one.
Under this assumption, two calibrated subactions differ by a
constant (this follows from proposition 5 in \cite{BLT}). In
particular the function $R_{+}:= \beta(f) + V\circ\sigma - V - f$ is
well defined. The function $R_{-}:=-R_{+}$ is the unique
accumulation point of $g_{c}/c$, on the uniform topology, so
$g_{c}/c \to R_{-}$ uniformly.

We denote by $|g|_{\theta}$ the Lipschitz constant of a Lipschitz function $g$.

\begin{lemma}
The function $R_{+}^{\infty}$ is lower semi-continuous.
\end{lemma}

\begin{proof}
We take $z, z_{j} \in X$, with $z_{j}\to z$. We will show
that
\[\liminf_{j\to\infty}R_{+}^{\infty}(z_{j}) \geq R_{+}^{\infty}(z).\]
In the case $R_{+}^{\infty}(z) =0$ the result is clearly true.

\textbf{First case:} $R_{+}^{\infty}(z) =\infty.$
\newline
Given $M>0$, let $n$ be such that $R_{+}^{n}(z) > 2M$. We fix
$n_{1}$, such that,
$\frac{M}{2^{n_1}|R_+|_{\theta}}<1$. Let $n_{0}$ be
such that for $j\geq n_{0}$, we have $d(z_{j},z) <
\frac{M}{2^{n_1}2^{n}|R_+|_{\theta}}$. Then, for
$j\geq n_{0}$:
\begin{align*}
R_{+}^{n}(z_{j}) &\geq R_{+}^{n}(z) -|R_+|_{\theta}(d(z_{j},z)+...+d(\sigma^{n-1}(z_{j}),\sigma^{n-1}z))\\
&\geq 2M - |R_+|_{\theta}(\frac{M}{2^{n_1}|R_+|_{\theta}}) \geq M.
\end{align*}
It follows that
\[\liminf_{j\to\infty}R_{+}^{\infty}(z_{j}) \geq M.\]
Taking $M\to +\infty$, we get
\[\liminf_{j\to\infty}R_{+}^{\infty}(z_{j}) = +\infty .\]

\textbf{Second case:} $R_{+}^{\infty}(z) = M>0$.
\newline
Fixed $\varepsilon >0$, there exist $n$, such that, $R_{+}^{n}(z) >
M-\varepsilon/2$. Let $n_{0}$ be such that for $j\geq n_{0}$, then we
have $d(z_{j},z) <
\frac{\varepsilon}{2^{n+2}|R_+|_{\theta}}$. Therefore, for
$j\geq n_{0}$:
\[R_{+}^{n}(z_{j})  \geq (M-\varepsilon/2) - |R_+|_{\theta}\left(\frac{\varepsilon}{2|R_+|_{\theta}}\right) = M-\varepsilon.\]
Therefore,  we get
\[\liminf_{j\to\infty}R_{+}^{\infty}(z_{j}) \geq M-\varepsilon.\]
Taking $\varepsilon\to 0$:
\[\liminf_{j\to\infty}R_{+}^{\infty}(z_{j}) \geq M .\]
\end{proof}

\bigskip

\bigskip

We note that in \cite{BCLMS} it is proved that
$\frac{1}{c}\log(\beta_{c}) \to \beta(f)$. We denote
$\varepsilon_{c} = \log(\beta_{c})-c\beta(f)$. Then, we have
$\frac{\varepsilon_{c}}{c}\to 0$.

\begin{lemma}
\[\lim_{c,n\to\infty} \left(\frac{1}{c}\log((L_{cR_-}^{n}1)(x)) - \frac{n.\varepsilon_{c}}{c}\right) = 0,\]
in particular, for a fixed $k$:
 \[\lim_{c,n\to\infty} \frac{1}{c}\log((L_{cR_-}^{n}1)(x)) - \frac{1}{c}\log((L_{cR_-}^{n+k}1)(x)) = 0.\]
\end{lemma}

\begin{proof}
Let $a$ be an accumulation point of  $\frac{1}{c}\log((L_{cR_-}^{n}1)(x))
- \frac{n.\varepsilon_{c}}{c}$, when $c,n \to\infty$. Then, there
exists $c_{j},n_{j} \to \infty$, such that,
\[\lim_{j\to\infty}\left(\frac{1}{c_{j}}\log((L_{c_{j}R_-}^{n_{j}}1)(x)) -
\frac{n_{j}.\varepsilon_{c_{j}}}{c_{j}}\right) = a.\]

Following \cite{BCLMS} we can take a subsequence $\{j_{i}\}$ such
that $\displaystyle{ \frac{1}{c_{j_{i}}}\log(h_{c_{j_{i}}})}$
converges uniformly to a calibrated subaction $V$. So there exist
sequences $c_{i},n_{i}\to\infty$ such that:
\[\lim_{i\to\infty}\left(\frac{1}{c_{i}}\log((L_{c_{i}R_-}^{n_{i}}1)(x)) -
\frac{n_{i}.\varepsilon_{c_{i}}}{c_{i}}\right) = a ,\, \text{ and } ,\,
\lim_{i\to\infty}\frac{1}{c_{i}}\log(h_{c_{i}}) = V.\] %
Denoting $\log(h_{c_{i}}) = c_{i}V + \delta_{c_{i}}$ where $|\delta_{c_{i}}|_{\infty}/c_{i} \to 0$, we have:
\begin{align*}
0 &= \lim_{i\to\infty} \frac{1}{c_{i}}\log((L_{g_{c_{i}}}^{n_{i}}1)(x))\\
&= \lim_{i\to\infty} \frac{1}{c_{i}}\log(\int_{\sigma^{n_{i}}(z)=x}e^{c_{i}f^{n_{i}}(z)+ \log(h_{c_{i}}(z)) - \log(h_{c_{i}}(x)) - n_{i}\log(\beta_{c_{i}})}\, dm ) \\
&= \lim_{i\to\infty} \frac{1}{c_{i}}\log(\int_{\sigma^{n_{i}}(z)=x}e^{c_{i}f^{n_{i}}(z)+ c_{i}V(z) - c_{i}V(x) - n_{i}c_{i}\beta(f) +\delta_{c_{i}}(z)-\delta_{c_{i}}(x) -n_{i}\varepsilon_{c_{i}}}\, dm) \\
&= \lim_{i\to\infty} \frac{1}{c_{i}}\log(\int_{\sigma^{n_{i}}(z)=x}e^{c_{i}R_-^{n_{i}}(z) +\delta_{c_{i}}(z)-\delta_{c_{i}}(x) -n_{i}\varepsilon_{c_{i}}}\, dm) \\
&=\lim_{i\to\infty} \frac{1}{c_{i}}\log(\int_{\sigma^{n_{i}}(z)=x}e^{c_{i}R_-^{n_{i}}(z) - n_{i}\varepsilon_{c_{i}}}\, dm) \\
&=\lim_{i\to\infty} \left(\frac{1}{c_{i}}\log(\int_{\sigma^{n_{i}}(z)=x}e^{c_{i}R_-^{n_{i}}(z)}\, dm) - \frac{n_{i}\varepsilon_{c_{i}}}{c_{i}}\right) = a.
\end{align*}
This shows that any accumulation point have to be equal to zero.
\end{proof}

\begin{center}
\textbf{The first inequality of Theorem \ref{ldpruelle1}}
\end{center}

\begin{proposition}\label{primeiroldp} For any closed set $F\subseteq X$
 \[\limsup_{c,n \to \infty} \frac{1}{c}\log((L_{g_{c}}^{n}\chi_{F})(x)) \leq  \sup_{z\in F}R_{-}^{\infty}(z)=-\inf_{x\in F}R_{+}^{\infty}(x).\]
 \end{proposition}

\begin{proof}

For a fixed $k$, we have that
\begin{align*}
&\limsup_{c,n\to\infty} \frac{1}{c}\log((L_{g_{c}}^{n+k}\chi_F)(x)) = \limsup_{c,n\to\infty} \frac{1}{c}\log(\frac{(L_{g_{c}}^{n+k}\chi_F)(x)}{(L_{g_{c}}^{n+k})(x)})\\
&=  \limsup_{c,n\to\infty} \frac{1}{c}\log(\frac{\int_{\sigma^{n+k}(z)=x}e^{cR_{-}^{n+k}(z) }\chi_F(z)\, dm}{ \int_{\sigma^{n+k}(z)=x}e^{cR_{-}^{n+k}(z)} \, dm}) \\
&= \limsup_{c,n\to\infty} \frac{1}{c}\log(\frac{\int_{\sigma^{n+k}(z)=x}e^{cR_{-}^{n+k}(z) }\chi_F(z)\, dm}{ \int_{\sigma^{n}(y)=x}e^{cR_{-}^{n}(y)} \, dm}) \\
&= \limsup_{c,n\to\infty} \frac{1}{c}\log(\frac{\int_{\sigma^{n}(y)=x}(e^{cR_{-}^{n}(y) }\int_{\sigma^{k}(z)=y}e^{cR_{-}^{k}(z)}\chi_{F}(z)\, dm)\, dm}{ \int_{\sigma^{n}(y)=x}e^{cR_{-}^{n}(y)}\, dm } ).
\end{align*}
Note, however that
\[\int_{\sigma^{k}(z)=y}e^{cR_{-}^{k}(z)}\chi_F(z) \, dm \leq e^{c\sup_{z\in F}R_{-}^{k}(z)},\]
then,
\[\limsup_{c,n\to\infty} \frac{1}{c}\log((L_{g_{c}}^{n}\chi_F)(x)) \leq  \limsup_{c,n\to\infty} \left( \sup_{z\in F}R_{-}^{k}(z)\right) = \sup_{z\in F}R_{-}^{k}(z).\]

For each $k$ fixed, we have that $R_{-}^{k}$ is a continuous
function, and $F\subset X$ is a compact set, then, there exist
$y_{k}\in F$, such that, $\sup_{z\in F}R_{-}^{k}(z) =
R_{-}^{k}(y_{k})$. Define
\[Y_{k}:= \{y \in F: \limsup_{c,n\to\infty} \frac{1}{c}\log((L_{g_{c}}^{n}\chi_F)(x)) \leq R_{-}^{k}(y)\}.\]
Then, $Y_{k}$ is closed (because $R_{-}^{k}$ is a continuous function) and not empty (because $y_{k}\in Y_{k}$). Using the fact that $R_{-} \leq 0$ we have
\[Y_{1} \supseteq Y_{2}\supseteq ...\]
These sets are closed and not empty, then there exist some $x_{0}
\in \bigcap_{k\geq 1}Y_{k}$. So, for each $k$:
\[\limsup_{c,n\to\infty} \frac{1}{c}\log((L_{g_{c}}^{n}\chi_F)(x)) \leq R_{-}^{k}(x_{0}).\]
Using the fact that $R_{-}^{k}(x_{0})\to R_{-}^{\infty}(x_{0})$, we
conclude that
\[\limsup_{c,n\to\infty} \frac{1}{c}\log((L_{g_{c}}^{n}\chi_F)(x)) \leq R_{-}^{\infty}(x_{0}) \leq \sup_{z\in F}R_{-}^{\infty}(z).\]
\end{proof}

\bigskip

\begin{center}\textbf{The second inequality of Theorem \ref{ldpruelle1}}
\end{center}

Suppose that $A$ is an open set. So, there exists $n_{0}$, such that, for
$n\geq n_{0}$, and, $x\in X$, there exists $y \in A$, such that,
$\sigma^{n}(y)=x$. More precisely, given $y=y_{1}y_{2}...$ in $A$,
let $\epsilon>0$, such that, $B(y,\epsilon)\subset A$. Now let $n_{0}$
such that $\frac{1}{2^{n_0}}<\epsilon$. If $z\in X$
coincide with $y_{1}...y_{n_{0}}$ in its first symbols, then $d(z,y)\leq \frac{1}{2^{n_{0}+1}}+\frac{1}{2^{n_{0}+2}}+...=
\frac{1}{2^{n_{0}}}<\epsilon$, and, so $z\in A$. We
conclude that given $x\in X$, we can get an  $y_{1}...y_{n_{0}}x \in
A$.

\begin{lemma} \label{maiorque}
There exist $y_{0}\in X$ such that
 \[\liminf_{c,n \to \infty} \frac{1}{c}\log((L_{g_{c}}^{n}\chi_A)(x)) \geq  \limsup_{k\to\infty} \left(\sup_{z\in A,\, \sigma^{k}(z)=y_{0}}R_{-}^{k}(z)\right).\]
  \end{lemma}
\begin{proof} For a fixed $k\geq n_{0}$, we have
\[\liminf_{c,n\to\infty} \frac{1}{c}\log((L_{g_{c}}^{n+k}\chi_A)(x)) = \liminf_{c,n\to\infty} \frac{1}{c}\log(\frac{\int_{\sigma^{n}(y)=x}e^{cR_{-}^{n}(y) }(\int_{\sigma^{k}(z)=y}e^{cR_{-}^{k}(z)}\chi_A(z)\, dm)\, dm}{ \int_{\sigma^{n}(y)=x}e^{cR_{-}^{n}(y)} \, dm} )\]
\[\geq \liminf_{c,n\to\infty} \frac{1}{c}\log(\inf_{y\in X}\int_{\sigma^{k}(z)=y}e^{cR_{-}^{k}(z)}\chi_A(z)\, dm)\]
\[= \liminf_{c\to\infty} \inf_{y\in X}\frac{1}{c}\log(\int_{\sigma^{k}(z)=y}e^{cR_{-}^{k}(z)}\chi_A(z)\, dm).\]
Then, we get
\[\liminf_{c,n\to\infty} \frac{1}{c}\log((L_{g_{c}}^{n}\chi_A)(x)) \geq \limsup_{k\to\infty}\liminf_{c\to\infty}\inf_{y\in X}\frac{1}{c}\log(\int_{\sigma^{k}(z)=y}e^{cR_{-}^{k}(z)}\chi_A(z)\, dm).\]

Let $y_{c,k}$ be such that
\[\inf_{y\in X}\frac{1}{c}\log(\int_{\sigma^{k}(z)=y}e^{cR_{-}^{k}(z)}\chi_A(z)\, dm) > \frac{1}{c}\log(\int_{\sigma^{k}(z)=y_{c,k}}e^{cR_{-}^{k}(z)}\chi_A(z)\, dm) -\frac{1}{k}.\]

 Then, we have:
\[\liminf_{c,n\to\infty} \frac{1}{c}\log((L_{g_{c}}^{n}\chi_A)(x)) \geq \limsup_{k\to\infty}\liminf_{c\to\infty}\frac{1}{c}\log(\int_{\sigma^{k}(z)=y_{c,k}}e^{cR_{-}^{k}(z)}\chi_A(z) \, dm)-\frac{1}{k}\]
\begin{equation}
= \limsup_{k\to\infty}\liminf_{c\to\infty}\frac{1}{c}\log(\int_{\sigma^{k}(z)=y_{c,k}}e^{cR_{-}^{k}(z)}\chi_A(z) \, dm). \label{3.1}
\end{equation}

As $X$ is a compact set, let $y_{0}$ be an accumulation point of
$y_{c,k}$, when $c,k\to\infty$. For $k$ sufficiently large, let $z_{k} \in A$, such that:
$\sigma^{k}(z_{k})=y_{0}$, and
\[R_{-}^{k}(z_{k}) > \sup_{z\in A,\, \sigma^{k}(z)=y_{0}}R_{-}^{k}(z) - \frac{1}{2k}.\]
We denote $z_{k}:=x_{k}...x_{1}y_{0}, \, x_{i}\in M$. We can take $\epsilon>0$ sufficiently small such that the ball $A_{k,\epsilon}:=\{a_{k}...a_{1}\in M^{k}: |(a_{k}...a_{1})-(x_{k}...x_{1})| \leq \epsilon \} $  satisfies:\newline
1. $a_{k}...a_{1}y_{0} \in A$, \newline
2. $a_{k}...a_{1}y_{c,k} \in A$, for $k,c>>0$, \newline
3. $R_{-}^{k}(a_{k}...a_{1}y_{0}) > \sup_{z\in A,\, \sigma^{k}(z)=y_{0}}R_{-}^{k}(z) - \frac{1}{k}$. \newline

Then, we get
\[ \limsup_{k\to\infty}\liminf_{c\to\infty}\frac{1}{c}\log(\int_{A_{k,\epsilon}}e^{cR_{-}^{k}(a_{k}...a_{1}y_{0})} \, dm)\]
\[\geq  \limsup_{k\to\infty}\liminf_{c\to\infty}\left(\sup_{z\in A,\, \sigma^{k}(z)=y_{0}}R_{-}^{k}(z) - \frac{1}{k} + \frac{1}{c}\log(m(A_{k,\epsilon}))\right) \]

\begin{equation}
= \limsup_{k\to\infty} (\sup_{z\in A,\, \sigma^{k}(z)=y_{0}}R_{-}^{k}(z)). \label{3.2}
\end{equation}
On the other hand, on $A_{k,\epsilon}$ we have:

\[R_{-}^{k}(a_{k}...a_{1}y_{c,k}) \geq R_{-}^{k}(a_{k}...a_{1}y_{0}) - |R_{-}|_{\theta}(\frac{1}{2} +\frac{1}{2^{2}} +...+ \frac{1}{2^{k}})d(y_{c,k},y_{0}) \]
\[\geq R_{-}^{k}(a_{k}...a_{1}y_{0}) - |R_{-}|_{\theta}d(y_{c,k},y_{0}).\]
Then, we get
\[\limsup_{k\to\infty}\liminf_{c\to\infty}\frac{1}{c}\log(\int_{\sigma^{k}(z)=y_{c,k}}e^{cR_{-}^{k}(z)}\chi_A(z) \, dm)\]
\[\geq \limsup_{k\to\infty}\liminf_{c\to\infty}\frac{1}{c}\log(\int_{A_{k,\epsilon}}e^{cR_{-}^{k}(a_{k}...a_{1}y_{c,k})} \, dm)\]
\[\geq \limsup_{k\to\infty}\liminf_{c\to\infty}\frac{1}{c}\log(\int_{A_{k,\epsilon}}e^{cR_{-}^{k}(a_{k}...a_{1}y_{0}) - c|R_{-}|_{\theta}d(y_{c,k},y_{0})} \, dm)\]
\begin{equation}
=\limsup_{k\to\infty}\liminf_{c\to\infty}\frac{1}{c}\log(\int_{A_{k,\epsilon}}e^{cR_{-}^{k}(a_{k}...a_{1}y_{0})} \, dm). \label{3.3}
\end{equation}

Using (\ref{3.1}), (\ref{3.3}) and (\ref{3.2}), we finish the proof.
\end{proof}

Now we fix the point $y_{0}$ given above. The next result is
basically contained in the proof of proposition 5 in $\cite{BLT}$.

\begin{lemma}\label{acumula}
Let $p$ be a point on the support of $\mu_{\infty}$. Let $y_{n}$ a
sequence satisfying $\sigma(y_{n}) = y_{n-1}, n=1,2,3,...$, and,
$0=R_{-}(y_{1})=R_{-}(y_{2})=...$ (it follows from the property of
the calibrated subaction). Then, $p$ is a accumulation point of
$\{y_{n}\}$.
\end{lemma}

\begin{proof}
Let $B$ be the set of accumulation points of $\{y_{n}\}$. B is
closed and $\sigma(B) = B$. Then there exists a invariant
probability $\nu$ with support on $B$. The inclusion $B\subseteq X$
implies the existence of an extension of $\nu$ to $X$ by the rule:
$\nu(\phi) := \nu(\phi.\chi_B)$. Using the fact that $R_{-}$ is a
continuous function, and, that $R_{-}(y_{n})=0, n=1,2,...$, we
conclude that $\chi_{B}.R_{-}  = 0$. Therefore, $\nu(R_{-}) = 0$, and then,
$\nu = \mu_{\infty}$. From this we get that  the support of $\mu_{\infty}$ is contained on $B$.%
\end{proof}

\bigskip

The next lemma follows the same reasoning of Lemma 18 in
$\cite{LMST}$:

\begin{lemma}\label{LMRT}
If $R_{-}^{\infty}(z) > -\infty$, then the family of probabilities
$\nu_{n}$ (also called empirical measures), given by $\phi \to
\frac{1}{n}\sum_{j=0}^{n-1}\phi(\sigma^{j}(z))$, converges to
$\mu_{\infty}$ weakly*, when $n\to \infty$.
\end{lemma}
\begin{proof}

Note that any  measure which in accumulation point of $\nu_{n}$ is an invariant probability.
We are going to show that
\[\liminf_{n\to\infty}\frac{1}{n}\sum_{j=0}^{n-1}f(\sigma^{j}(z)) \geq \mu_{\infty}(f).\]
Let $M=R_{-}^{\infty}(z)$. Then, for each $n$ we have
$R_{-}^{n}(z)\geq M$, so:
\[V(z) - V(\sigma^{n}(z)) -n\mu_{\infty}(f) + \sum_{j=0}^{n-1}f(\sigma^{j}(z))= \sum_{j=0}^{n-1}R_{-}(\sigma^{j}(z)) =R_{-}^{n}(z)  \geq M.\]
Then, we get
\[\frac{1}{n}\sum_{j=0}^{n-1}f(\sigma^{j}(z)) \geq \frac{M}{n} -\frac{2|V|_{\infty}}{n} + \mu_{\infty}(f).\]
Finally, taking $\displaystyle{\liminf_{n\to\infty}}$ in the above expression we get our claim.%
\end{proof}

\begin{corollary}
If $R_{-}^{\infty}(z) > -\infty$, and, $p \in supp(\mu_{\infty})$,
then $p$ is an accumulation point of ${\sigma^{n}(z)}$.
\end{corollary}

\begin{proof}
Let $p \in supp(\mu_{\infty})$, and $\varepsilon>0$. Consider the
ball $B(p,\varepsilon):=\{x\in X: d(x,p)<\varepsilon\}$. Using the fact that $p \in supp(\mu_{\infty})$, we have that
$\mu_{\infty}(B(p,\varepsilon))>0$. Therefore, by the above lemma, we have
that $\{\sigma^{n}(z)\}$ is in this ball for infinite values of $n$.
\end{proof}

\begin{lemma}
\[\sup_{z \in A} R_{-}^{\infty}(z) \leq \limsup_{k\to\infty}
(\sup_{z\in A,\, \sigma^{k}(z)=y_{0}}R_{-}^{k}(z)).\]
\end{lemma}

\begin{proof}
We fix a point $p\in supp(\mu_{\infty})$, and, we denote
$p=p_{1}p_{2}...$. For $n\geq n_{0}$, there exists $y\in A$, such that
$\sigma^{n}(y)=p.$ Note that
$R_{-}^{\infty}(y)=R_{-}^{n}(y)>-\infty.$ Therefore, $\sup_{z \in A}
R_{-}^{\infty}(z)>-\infty$. Let $z_{0}\in A$ be such that,
$R_{-}^{\infty}(z_{0})>-\infty$, and, denote $z_{0}=x_{1}x_{2}...$.
Given $t\in \mathbb{N}$, let $n(t)$ be such that,
$d(\sigma^{n(t)}(z_{0}),p) \leq \frac{1}{2^{t}}$, and moreover, such
that, the choice $x_{1}x_{2}...x_{n(t)}$,  determines that $z_{0}\in
A$ (open).
\newline
From Lemma $\ref{acumula}$ there exists a pre-image of $y_{0}$ (we
assume  it is a point  $y_{l(t)}$ of the form $a_{l(t)}...a_{1}(y_{0})$),
such that,
 $R_{-}^{l(t)}( y_{l(t)})=0$, and, $d(y_{l(t)},p)<\frac{1}{2^{t}}$.
Define $z(t)$ by $$z(t):= x_{1}...x_{n(t)}a_{l(t)}...a_{1}(y_{0}).$$
Then,
\begin{align*}
\sup_{z\in A: \sigma^{l(t)+n(t)}(z)=y_{0}}& R_{-}^{l(t)+n(t)}(z) \geq R_{-}^{l(t)+n(t)}(z(t))\\
&= R_{-}^{n(t)}(z(t)) + R_{-}^{l(t)}(y_{l(t)}) = R_{-}^{n(t)}(z(t))\\
&\geq R_{-}^{n(t)}(z_{0}) - |R_{-}|_{\theta}2(\frac{1}{2^{t}}+\frac{1}{2^{t+1}}+...+\frac{1}{2^{t+n(t)}}) \\
&\geq R_{-}^{n(t)}(z_{0}) - \frac{|R_{-}|_{\theta}}{2^{t}} \geq R_{-}^{\infty}(z_{0}) - \frac{|R_{-}|_{\theta}}{2^{t}}.
\end{align*}

When, $t\to\infty$, we get
\[\limsup_{k\to\infty} (\sup_{z\in A,\, \sigma^{k}(z)=y_{0}}R_{-}^{k}(z)) \geq \limsup_{t\to\infty} \sup_{z\in A: \sigma^{l(t)+n(t)}(z)=y_{0}} R_{-}^{l(t)+n(t)}(z)\]
\[\geq R_{-}^{\infty}(z_{0}).\]
Using the fact that $z_{0}$ is arbitrary, and satisfies
$R_{-}^{\infty}(z_{0})>-\infty$, we finish the proof.
\end{proof}

From this result and lemma \ref{maiorque} we get:
\begin{proposition}
\[\liminf_{c,n\to\infty} \frac{1}{c}\log((L_{g_{c}}^{n}\chi_A)(x)) \geq \sup_{z\in A}R_{-}^{\infty}(z)= - \inf_{z\in A} (R_{+}^{\infty}(z)).\]
\end{proposition}

This finish the proof of Theorem \ref{ldpruelle1}.

\end{document}